\documentclass[reqno, 11pt]{amsart}
\usepackage[colorlinks]{hyperref}
\usepackage{url}
\usepackage{times}

\newtheorem{theorem}{Theorem}[section]
\newtheorem{lemma}[theorem]{Lemma}

\theoremstyle{definition}

\newtheorem{remark}[theorem]{Remark}
\newtheorem{question}[theorem]{Question}

\numberwithin{equation}{section}

\usepackage{bbm}
\usepackage{url}
\usepackage{euscript}
\usepackage{pb-diagram}
\usepackage{lamsarrow}
\usepackage{pb-lams}
\usepackage{amsmath}
\usepackage{amsthm}

\usepackage{amsxtra}
\usepackage{amssymb}
\usepackage{pifont}
\usepackage{amsbsy}

\usepackage{graphicx}
\usepackage{epstopdf}

\newskip\aline \newskip\halfaline
\def\skipaline{\vskip\aline}

\def\qedbox{$\rlap{$\sqcap$}\sqcup$}
\def\qed{\nobreak\hfill\penalty250 \hbox{}\nobreak\hfill\qedbox\skipaline}

\def\proofend{\eqno{\mbox{\qedbox}}}

\newcommand\bR{{\mathbb R}}

\newcommand{\bV}{{{\mathbb V}}}

 \DeclareMathOperator{\Hom}{Hom}

\newcommand{\be}{\boldsymbol{e}}
\newcommand{\bsf}{\boldsymbol{f}}

\newcommand{\br}{\boldsymbol{r}}

\newcommand{\bu}{{\boldsymbol{u}}}
\newcommand{\bv}{\boldsymbol{v}}
\newcommand{\bw}{\boldsymbol{w}}
\newcommand{\bx}{{\boldsymbol{x}}}

\newcommand{\bsE}{\boldsymbol{E}}

\newcommand{\bsS}{\boldsymbol{S}}

\newcommand{\bsU}{{\boldsymbol{U}}}
\newcommand{\bsV}{{\boldsymbol{V}}}
\newcommand{\bsW}{{\boldsymbol{W}}}

\newcommand{\bnu}{{\boldsymbol{\nu}}}

\newcommand{\bsi}{\boldsymbol{\sigma}}

\newcommand{\bom}{\boldsymbol{\omega}}

\newcommand{\si}{{\sigma}}

\newcommand{\eP}{\EuScript{P}}
\newcommand{\eQ}{\EuScript{Q}}

\newcommand{\eS}{\EuScript{S}}
\newcommand{\eT}{\EuScript{T}}

\newcommand{\eV}{\EuScript{V}}

\newcommand{\ra}{\rightarrow}
\newcommand{\hra}{\hookrightarrow}
\newcommand{\Lra}{{\longrightarrow}}

\newcommand{\lan}{\langle}
\newcommand{\ran}{\rangle}

\def\inpr{\mathbin{\hbox to 6pt{\vrule height0.4pt width5pt depth0pt \kern-.4pt \vrule height6pt width0.4pt depth0pt\hss}}}

\newcommand{\pa}{\partial}

\newcommand{\dual}{{\spcheck{}}}

\begin{document}

\title{Statistics of linear  families  of  smooth functions on knots}

\date{Started  June 2, 2010. Completed  on June 4, 2010.
Last modified on {\today}. }

\author{Liviu I. Nicolaescu}

\address{Department of Mathematics, University of Notre Dame, Notre Dame, IN 46556-4618.}
\email{nicolaescu.1@nd.edu}
\urladdr{\url{http://www.nd.edu/~lnicolae/}}

\begin{abstract}   Given a   knot $K$ in an Euclidean space $\bR^n$,  and   a     finite dimensional subspace $\bsV\subset C^\infty(K)$,  we    express the expected number of critical  points of  a random function in $\bsV$in terms of  an integral-geometric invariant of $K$ and $\bsV$.  When $\bsV$  consists of  the restrictions   to $K$ of homogeneous polynomials  of degree $\ell$  on $\bR^n$,  this invariant  takes the form  of total curvature of  a certain immersion of $K$.  In particular, when $K$ is the unit circle in $\bR^2$ centered at the origin,  then the expected number of critical points of the restriction to $K$ of a random  homogeneous polynomial  of  degree $\ell$  is $2\sqrt{3\ell-2}$, and the expected number of critical points on $K$ of a random  trigonometric polynomial of degree $k$ is approximately $1.549 k$.  \end{abstract}

\maketitle

\begin{center}
\textit{To the memory of my mathematical  hero, Vladimir Igorevich Arnold}
\end{center}

\tableofcontents

\section*{Introduction}
\setcounter{equation}{0}

A celebrated   result of F\'{a}ry and Milnor \cite{Fa, Mil} states that the expected  number of critical points  of the restriction  to a knot $K\hra \bR^3$ of a  random linear map  $h:\bR^3\ra \bR$  is equal to an integral-geometric  invariant of the knot, namely, its (suitably normalized) total curvature.     It is natural to ask how this  result  changes if, instead of  random linear maps, we look  at  random  homogeneous polynomials $P:\bR^3\ra \bR$ of a given degree $\ell$. It is convenient to investigate an even more general situation.

Suppose that $\bsU$ is an oriented   real Euclidean space  of dimension $n$  and $K\hra \bsU$ is a knot in $\bsU$, i.e., an smoothly embedded $S^1$.  We assume that $0\not\in K$.  For any positive integer $\ell$ we denote by $\eQ_\ell(\bsU)$ the space of symmetric $\ell$-linear forms  on $\bsU$. Each such form  defines a polynomial  function
\[
P_\Phi:\bsU\ra \bR,\;\; P_\Phi(\bu)=\Phi(\underbrace{\bu,\dotsc,\bu}_\ell).
\]
We denote by $f_\Phi$ the restriction     of $P_\Phi$ to the knot $K$.  We  will refer to such functions as   \emph{polynomial functions on $K$ of  degree  $\ell$}. For  random $\Phi$, the function   $f_\Phi$ is Morse,  and we denote by $\mu_K^\ell(\Phi)$ its number of critical points.      We can regard $\mu_K(\Phi)$ as a random variable and  ask how  is  its expectation  related to the global geometry of $K$.

To describe this relationship,  denote by $\eS^\ell$ the unit sphere in $\eQ_\ell(\bsU)$ with respect to the natural metric $\lan-,-\ran_\ell$ induced by the Euclidean metric on  $\bsU$. The expected  number of  critical points  of $f_\Phi$ is the real number $\mu_K^\ell$ defined by
\[
\mu_K^\ell=\frac{1}{{\rm area}\,(\eS^\ell)} \int_{\eS^\ell} \mu_K^\ell(\Phi)\,dS(\Phi).
\]
In  Theorem \ref{th: vero} we express  $\mu_K^\ell$ as an integral geometric invariant of $K$.     This invariant is  described in terms of the Veronese  map $\eV_\ell :\bsU\ra \eQ_\ell(\bsU)$  uniquely determined by the equality 
\[
\bigl\lan\,\Phi, \eV_\ell(\bu)\,\bigr\ran_\ell=\Phi(\underbrace{\bu,\dotsc,\bu}_\ell),\;\;\forall \bu\in\bsU,\;\;\Phi\in\eQ_\ell(\bsU).
\]
The restriction of the Veronese map to $K$ produces an immersion
\[
\eV_\ell: K\ra \eQ_\ell(\bsU)
\]
that we called the \emph{$\ell$-th Veronese immersion} of $K$.    In Theorem \ref{th: vero} we show that $\mu_K^\ell$ is the total curvature of the $\ell$-th Veronese immersion.       The case $\ell=1$ of this theorem is the celebrated  result of F\'{a}ry and Milnor \cite{Fa, Mil}.

As an application we compute $\mu_K^\ell$, when $K$ is the unit circle in $\bR^2$ centered at the origin. More precisely, we show that  in this case $\mu^\ell_K=2\sqrt{3\ell-2}$. In other words,  the  expected  number of critical points  of a polynomial function on $K$ of degree $\ell$ is $2\sqrt{3\ell-2}$. 

We obtain Theorem  \ref{th: vero} as a special case of  Theorem \ref{th: main}     that describes the   expected    number of critical points  of a random  function belonging to a fixed finite dimensional subspace $\bsV\subset C^\infty(K)$ satisfying a nondegeneracy condition (\ref{eq: nondeg}).

Theorem \ref{th: main} has another interesting consequence. More precisely, in Theorem \ref{th: trig} we show that the expected number of critical points on the unit circle of a random     trigonometric polynomial of degree $n$ is
\[
T_n=2\sqrt{\frac{3\beta_5(n+1)}{5\beta_3(n+1)}},
\]
where $\beta_k$ denotes the $k$-th Bernoulli polynomial. In particular
\[
T_n\sim 2n\sqrt{\frac{3}{5}} \;\;\mbox{as $n\ra \infty$}.
\]

\medskip

\noindent \ding{45} {\bf Notations.} We will  denote by $\bsi_n$ the ``area'' of the  round  $n$-dimensional sphere $S^n$ of radius $1$, and by $\bom_n$  the  ``volume'' of the  unit ball in $\bR^n$.  These quantities are uniquely determined by the equalities (see \cite[Ex. 9.1.11]{N1})
\[
\bsi_{n-1}=n\bom_n,\;\;\bom_n=\frac{\Gamma(1/2)^n}{\Gamma(1+n/2)},\;\;\Gamma(1/2)=\sqrt{\pi},
\]
where $\Gamma$  is Euler's Gamma function.

\section{An   abstract result}
\label{s: 1}
\setcounter{equation}{0}

Let $\bsU$ and $\bsV$ be a real, oriented Euclidean space of dimension $n$.  We denote by $(-,-)_\bsU$ the inner product in $\bsU$ and by $S(\bsU)$ the unit sphere in $\bsU$. Suppose that  $K\hra \bsU$ is a smooth knot in $\bsU$, i.e.,   the image of a smooth embedding  $S^1\hra \bsU$.  Fix  an arclength parametrization of $K$
 \[
 s\mapsto \bx(s),\;\; 0\leq s\leq L,
 \]
 where $L$ denotes the length of $K$. 
 
 Assume that we are given a  finite dimensional subspace $\bsV\subset  C^\infty(\bsU)$ of dimension  $N$ satisfying the nondegeneracy condition
 \begin{equation}
 \forall s\in [0,L],\;\;\exists  \bv\in\bsV:\;\; d_{\bx(s)}\bv (\bx'(s))\neq 0,
 \label{eq: nondeg}
 \end{equation}
 where $d_\bx \bv $ denotes the differential of the function  $\bv$ at $\bx$.   We fix an  Euclidean inner product $(-,-)_\bsV$ on $\bsV$ and we denote by $S(\bsV)$ the unit sphere in $\bsV$ centered at the origin. 
 
 As explained in \cite[\S 1.2]{N1}, the  condition (\ref{eq: nondeg}) implies that for generic $\bv\in\bsV$, the restriction of the function $\bv$ to   $K$ is a Morse function. We denote  by $\mu_K(\bv)$  its number of critical points.      Note that for any  nonzero scalar $t\in\bR$ the restriction of $t\bv$ to $K$ is Morse if and only if    the restriction of $\bv$ to $K$ is such.     Thus, it suffices to concentrate on functions $\bv\in S(\bsV)$.   The main goal    of this section  is the computation of the expected number of critical points of a random function in $S(\bsV)$, i.e., the quantity
 \[
 \mu_K^\bsV:=\frac{1}{{\rm area}\,\bigl(\,S(\bsV)\,\bigr)}\int_{S(\bsV)} \mu_K(\bv)\,|dS(\bv)|=\frac{1}|{\bsi_{N-1} \bigl(\,S(\bsV)\,\bigr)}\int_{S(\bsV)} \mu_K(\bv)\,|dS(\bv)|,
 \]
where  $|dS|$  is the ``area'' density on   $S(\bsV)$.
 
 To describe the result we need  to introduce the main  characters.  Set $\bsV\dual:=\Hom(\bsV,\bR)$ and observe that we have a smooth map $\xi: K\ra \bsV\dual$ that associates to   each point $\bx(s)\in K$ the linear map
 \begin{equation}
 \xi_s:\bsV\ra \bR,\;\ \; \bv\mapsto d_{\bx(s)}\bv(\bx'(s))\in \bR.
 \label{eq: xi}
 \end{equation}
 Using the metric induced isomorphism $\bsV\dual\ra \bsV$ we obtain a dual map 
 \begin{equation}
 \xi^\dag: K\ra \bsV. 
 \label{eq: xi-d}
 \end{equation}
 The nondegeneracy condition (\ref{eq: nondeg}) implies  that $\xi^\dag_s\neq 0$,  $\forall s$. We  can thus define a
 smooth map
 \[
 \bnu: K\ra S(\bsV),\;\;\bnu(s):=\frac{1}{|\xi^\dag_s|}\xi_s.
 \]
 \begin{theorem} Let $\bsU$, $K$ and $\bsV$ be as above then
 \[
 \mu_K^\bsV:=\frac{1}{\bsi_{N-1}}\int_{S(\bsV)} \mu_K(\bv)\, |dS(\bv)|= \frac{1}{\pi}\int_K |\nu'(s)|\, |ds|.
 \]
 \label{th: main}
 \end{theorem}
 
\begin{proof} Set
 \[
 E_\bnu:=\bigl\{\, (\bx,\bv)\in K\times S(\bsV);\;\; \bnu(\bx)\perp \bv \,\bigr\}.
 \]
 Note that $E_\bnu$ can be  alternatively defined as the zero set of the function
 \[
 F: K\times S(\bsV)\ra \bR,\;\;F(\bx,\bv)=(\bnu(x),\bv)_\bsV,
 \]
 and  the differential of $F$ is nonzero along the level set  $\{F=0\}$. This shows that $E_\bnu$ is a smooth submanifold of  $K\times S(\bsV)$ and
 \[
 \dim E_\bnu=\dim S(\bsV)= N-1.
 \]
 We    denote by $g_E$ the  metric on $E_\bnu$ induced  by the natural metric of  $K\times S(\bsV)$, and by $|dV_E|$ the associated     volume density. We have two natural (left and right) smooth maps
 \[
 K\stackrel{\lambda}{\longleftarrow} E_{\bnu}\stackrel{\rho}{\Lra} \bsS(V).
 \]
 The fiber of $\lambda$ over $\bx\in K$ is the Equator
 \[
 E_\bx:=\bigl\{ \bv\in S(\bsV);\;\; \bv\perp \bnu(x)\,\bigr\},
 \]
 while the fibers of $\rho$ are generically finite.  Tautologically, the restriction of $\lambda$ to any fiber of $\rho$ is injective.         More importantly, for any $\bv\in S(\bsV)$ the subset $\lambda\bigl(\rho^{-1}(\bv))\subset K$ is the set of critical points of  the restriction of $\bv$ to $K$.   Hence,  for generic $\bv\in\bsV$ we have
 \[
 \mu_K(\bv)=\#\rho^{-1}(\bv).
 \]
 We  have thus reduced the problem to computing  the average  number  of points  in the fibers of $\rho$ in terms of  integral-geometric invariants of the map $\bnu$.      We will achieve this in (\ref{eq: av4}).

   The area   formula  (see \cite[\S 3.2]{Feder} or \cite[\S 5.1]{KP}) implies  that
\begin{equation}
\int_{S(\bsV)} \# \rho^{-1}(\bw) |dS(\bv)|=\int_E J_\rho(x,\bv) |dV_E(x,\bv)|,
\label{eq: av}
\end{equation}
where   the nonnegative  function $J_\rho$ is the Jacobian of $\rho$ defind by the equality
\[
\rho^*|dS|=J_\rho \cdot |dV_E|.
\]
To compute the integral in the right-hand side of (\ref{eq: av}) we need a  more explicit description of the geometry of $E_\bnu$.

Let $(\bx_0,\bv_0)\in E_\bnu$.  Fix an orthonormal frame $e_1,\dotsc,  e_{N-2}$ of the tangent space of $E_{\bx_0}$ at $\bv_0$.     Assume $\bx_0=\bx(s_0)$.  Then  the tangent space  of $E_\bnu$ at $(\bx_0,\bv_0)$ consists of tangent vectors
\[
\dot{s}\bx'(s_0)\oplus \dot{\bv}\in T_{\bx_0}K\oplus T_{\bv_0}S(\bsV),\;\;\dot{s}\in\bR, 
\]
such that
\begin{equation}
\dot{s}(\bnu'(s_0),\bv_0)_\bsV+ (\nu(s_0),\dot{\bv})_\bV=0.
\label{eq: tan}
\end{equation}
 Define
 \[
 e_0= \bx'(s) -(\bnu'(s_0),\bv_0)_\bsV \bnu(s_0).
 \]
For simplicity we set
\[
\mu(\bx_0, \bv_0):=\bigl|\, (\bnu'(s_0),\bv_0)_\bsV\,\bigr|.
\]
The equality (\ref{eq: tan}) implies that the collection
 \[
 e_0, e_1,\dotsc,  e_{N-2}
 \]
 is an orthogonal  basis of $T_{(\bx_0,\bv_0)}E_{\bnu}$. Moreover, the length of $e_0$ is 
 \[
 |e_0|=\sqrt{ 1+\mu(\bx_0,\bv_0)^2}.
 \]
 If we denote by $dS_\bx$ the  area form on $E_\bx$ then we see that at $(\bx_0, \bv_0)$ we have
 \[
 ds\wedge dS_\bx(e_0, e_1,\dotsc, e_{N-2})= \pm 1.
 \]
 Hence at $(\bx_0,\bv_0)$ we have
 \begin{equation}
 |dV_E|=|e_0|\cdot |ds\wedge dS_\bx|=\sqrt{ 1+\mu(\bx_0,\bv_0)^2}\bigl|ds\wedge dS_\bx|.
\label{eq: vol}
\end{equation}
The differential of  $\rho$ at $(\bx_0, \bv_0)$ is the linear map
 \[
 \rho_*: T_{(\bx_0,\bv_0)}E_\bnu\ra T_{\bv_0}S(\bsV)
 \]
 given by
 \[
 e_0\mapsto  (\bnu'(s_0),\bv_0)_\bsV\,\cdot\,  \bnu(s_0),\;\; e_i\mapsto e_i,\;\;1\leq i\leq N-2.
 \]
 We conclude that
 \begin{equation}
 J_\rho(\bx_0,\bv_0)=\frac{\mu(\bx_0,\bv_0)}{\sqrt{ 1+\mu(\bx_0,\bv_0)^2}}.
 \label{eq: jac}
 \end{equation}
 Using (\ref{eq: vol}), (\ref{eq: jac}) and the co-area formula  for the map $\lambda$ \cite[Prop. 9.1.8]{N0} we deduce 
 \begin{equation}
 \int_E J_\rho(x,\bv) |dV_E(x,\bv)|=\int_K \underbrace{\left(\int_{E_\bx} \mu(\bx, \bv) |dS_\bx(\bv)|\,\right)}_{=:J(\bx)} |ds(\bx)|.
 \label{eq: av1}
 \end{equation}
 
To proceed further we need the following elementary result.

 \begin{lemma} Suppose $\bsW$ is an $(m+1)$-dimensional oriented real Euclidean space with inner product $(-,-)$, and $\bv_0\in\bsW$. Denote by $S(\bsW)$ the unit  sphere in   $\bsE$, and by $|dS|$ the ``area'' density on $S(\bsW)$.  Then
 \[
I(\br_0) =\int_{S(\bsW)} |(\bv_0,\bw)| |dS(\bw)|=\bom_m|\bv_0|.
 \]
 \label{lemma: cos}
 \end{lemma}
 
 \begin{proof} Fix an orthonormal basis $(e_0,e_1,\dotsc, e_m)$ of $\bsW$ and denote by $(w_0,\dotsc, w_m)$ the resulting coordinates.   Observe that for any  orthogonal transformation $T: \bsW\ra \bsW$ we have $I(T\bv_0)=I(\bv_0)$ so that $I(\br_0)$ depends only on the length of $\bv_0$.  Thus, after an orthogonal transformation,  we can assume  $\bv_0= ce_0$. We can then write 
 \[
 I(ce_0)=|c|\int_{S(\bsW)}  |(e_0,\bw)| |dS(\bw)|= |c|\int_{S(\bsW)}  |w_0| |dS(\bw)|.
 \]
 Denote by $S_+(\bsW)$ the hemisphere $w_0>0$. We have
 \[
 I(ce_0)= 2|c|\int_{S_+(\bsW)}  w_0 |dS(\bw)|.
 \]
 The upper  hemisphere $S_+(\bsW)$ is the graph of the map
 \[
 w_0= \sqrt{1-r^2},\;\;r :=  \sqrt{w_1^2+\cdots +w_m^2}<1.
 \]
 Observe that 
 \[
 |\nabla w_0|=\frac{r^2}{1-r^2}\;\; \mbox{and}\;\; |dS(\bw)|=\frac{1}{\sqrt{1-r^2}} |dV_m|=\frac{1}{w_0}|dV_m|,
 \]
 where $|dV_m|$ denotes the Euclidean volume density on the $m$-dimensional Euclidean space with coordinates $(w_1,\dotsc, w_m)$. We deduce
 \[
 I(ce_0)= 2|c|\int_{r<1} |dV_m|=2|c|\bom_m.
 \]
 \end{proof}
 
  Using Lemma \ref{lemma: cos}  in the special case when $m+1=N-1$, $\bsW$ is the  hyperplane containing the Equator $E_{\bx_0}$, $\bx_0=\bx(s_0)\in K$,  and $\bv_0=\bnu'(s_0)$ we deduce
  \[
  J(\bx_0)=I(\bnu'(s_0))=2|\bnu'(s_0)|\bom_{N-2}.
  \]
  From (\ref{eq: av}) we now deduce
  \begin{equation}
  \frac{1}{\bsi_{N-1}}\int_{S(\bsV)} \# \rho^{-1}(\bv) |dS(\bv)|=\frac{2\bom_{N-2}}{\bsi_{N-1}} \int_K |\bnu'(s)|  |ds|.
  \label{eq: av2}
  \end{equation}
  Using the   equalities
  \[
  \bsi_{N-1}= N\bom_N,\;\;\bom_m=\frac{\Gamma(1/2)^m}{\Gamma(1+m/2)}.
  \]
  we deduce 
  \[
  \frac{2\bom_{N-2}}{\bsi_{N-1}}=\frac{2}{N\Gamma(1/2)^2} \frac{\Gamma(1+N/2)}{\Gamma(N/2)}=\frac{1}{\pi},
  \]
  where at the last step we have use the  classical identities  
  \[
  \Gamma(1+N/2)=\frac{N}{2}\Gamma(N/2),\;\;\Gamma(1/2)=\sqrt{\pi}.
  \]
  We have thus proved 
  \begin{equation}
   \frac{1}{\bsi_{N-1}}\int_{S(\bsV)} \# \rho^{-1}(\bv) |dS(\bv)|=\frac{1}{\pi}\int_K |\bnu'(s)|  |ds|.
   \label{eq: av4}
   \end{equation}
   \end{proof}
   
   \section{Polynomial  Morse functions on knots}
   \setcounter{equation}{0}

  Let $\bsU$ and $K\hra\bsU$ as above.   We  want to apply  the  results in the previous section to a special choice of $\bsV$.  Fix a positive integer $\ell$.   Denote by $\eQ_\ell(\bsU)$ the space of  symmetric $\ell$-linear  forms on $\bsU$, or equivalently, the space of homogeneous polynomials on $\bsU$ of degree $\ell$.   This will be our choice of subspace $\bsV\subset C^\infty(\bsU)$. The nondegeneracy assumption (\ref{eq: nondeg}) translates into the condition $0\not\in K$. Moreover,
  \[
  \dim\eQ_\ell(\bsU)=\binom{n+\ell-1}{\ell}.
  \]
  The   Euclidean metric on $\bsU$ induces an inner product $\lan-,-\ran_\ell$ on $\eQ_\ell(\bsU)$. Denote by $\eS^\ell$ the unit sphere in $\eQ_\ell(\bsU)$. For any $\Phi\in \eS^\ell$ we obtain a function
 \[
 f_\Phi:K\ra \bR,\;\; f_\Phi(\bx)=\Phi(\underbrace{\bx,\dotsc, \bx}_k).
 \]
 The function $f_\Phi$ is a Morse function for almost all $\Phi\in \eS^\ell$. We denote by $\mu_K(\Phi)$  the number of critical points of   $f_P$.   The main goal of this section is to describe the average
 \[
 \mu_K^\ell= \frac{1}{\bsi_{N_\ell-1}} \int_{\eS^\ell} \mu^\ell_K(\Phi) \,|dS(\Phi)|
 \]
 in terms of  integral-geometric  invariants of $K$.  We will achieve this   by reducing the problem to the situation investigated in the previous section.

 We begin with a linear algebra digression.    The $\ell$-th Veronese map is the linear map
 \[
 \eV=\eV_\ell: \underbrace{\bsU\otimes\cdots\otimes\bsU}_\ell\ra \eQ_\ell(\bsU),
 \]
 uniquely determined by the requirement
 \begin{equation}
\bigl\lan\, \Phi,\eV(\bu_1,\dotsc, \bu_\ell)\,\bigr\ran_\ell=\Phi(\bu_1,\dotsc,\bu_\ell),,\;\;\forall \bu_1,\dotsc,\bu_\ell \in\bsU,\;\;\Phi\in \eQ_\ell(\bsU).
 \label{eq: ver0}
 \end{equation}
 Note that for any smooth paths
 \[
 t\mapsto \bu_i(t)\in    \bsU,\;\;i=1,\dotsc, \ell,
 \]
 we have
 \[
 \frac{d}{dt}\eV(\bu_1,\dotsc,\bu_\ell)=\eV(\dot{\bu}_1,\bu_2,\dotsc,\bu_\ell)+\cdots +\eV(\bu_1,\dotsc,\bu_{\ell-1}, \dot{\bu}_\ell),
 \]
 where a  dot indicates $t$-differentiation.    For any  $\bu\in  \bsU$ we set
 \[
 M_\bu:=\eV(\underbrace{\bu,\dotsc,\bu}_\ell).
 \]
 With these notations we  deduce that
 \[
 f_\Phi(\bx)=\lan \Phi, M_\bx\ran_\ell,\;\;\forall \bx\in K,\;\;\Phi\in\eS^\ell.
 \]
 We deduce that $\bx_0=\bx(s_0)$ is a critical point of $f_\Phi$ if
 \[
 \lan \Phi, M'_{\bx(s_0)}\ran_\ell=0,
 \] 
  where a prime $'$ indicates $s$-differentiation.  Observe  that
 \[
 M'_{\bx(s)}=\ell \eV(\bx'(s),\underbrace{\bx(s),\dotsc,\bx(s)}_{\ell-1}),
 \]
 Since $\bx(s),\bx'(s)\neq 0$, $\forall s$ we deduce
 \begin{equation}
 M'_{\bx(s)}\neq 0,\;\;\forall s.
 \label{eq: immer}
 \end{equation}
Define 
 \[
\bnu_\ell:K\ra \eS^\ell,\;\; \bnu(s) =\frac{1}{|M'_{\bx(s)}|}M'_{\bx(s)}.
 \]
Using (\ref{eq: av4}) we deduce
 \begin{equation}
 \mu^\ell_K=\frac{1}{\pi} \int_K |\bnu_\ell'(s)| |ds|.
 \label{eq: av5}
 \end{equation}
 To formulate this in a more geometric fashion observe that (\ref{eq: immer}) implies that the map
 \begin{equation}
 K\ni\bx\mapsto \eV(\underbrace{\bx,\dotsc,\bx})\in\eQ_\ell
 \label{eq: vero}
 \end{equation}
 is an immersion.   We will refer to it  as the \emph{$\ell$-th Veronese immersion}  of  $K$.   For every $s$ the unit vector $\bnu_\ell(s)$ is the unit tangent vector field along this  Veronese embedding.   The integral    in the right-hand side of (\ref{eq: av5})  is then precisely the total curvature of the $\ell$-th Veronese immersion of $K$. We denote it by $\tau_\ell(K)$.  We have thus proved the following result.
 
 \begin{theorem} Let  $K\subset \bsU$ be a knot in the Euclidean space $\bsU$ such that $0\not \in K$. Then for any positive integer $\ell$ we have
 \[
 \mu_K^\ell=\frac{1}{\pi}\tau_\ell(K), 
 \]
 where $\tau_\ell(K)$ denotes  the  total curvature of the $\ell$-th Veronese immersion  of $K$ defined by (\ref{eq: vero}).\qed
 \label{th: vero}
 \end{theorem}
 
 \begin{remark}  (a)The  case $\ell=1$ in the above theorem  is precisely  the celebrated result of  F\'{a}ry and Milnor, \cite{Fa, Mil}.

 (b) The arguments in the proof   of  Theorem \ref{th: vero} show something more. More precisely, if $X\subset \bsU$ is a compact submanifold of $\bsU$, then  the  expected number of critical points  of $P_\Phi|_X$ is  equal to the expected number of critical points  of the function $\lambda \circ \eV_\ell|_X$, where $\lambda:\eQ_\ell\ra \bR$ is a random  linear function of norm $1$.  The arguments in Chern and Lashof, \cite{CL}, show that this number is the total curvature of the $\ell$-th Veronese immersion $\eV_\ell:X\ra \eQ_\ell(\bsU)$. \qed
 \end{remark}
 
 \section{An application}
 \setcounter{equation}{0}
 
 We want to compute the  higher total curvatures $\tau_\ell(K)$ when $K$ is the unit circle in $\bsU=\bR^2$ centered at the origin.   Denote by $K_\ell$ the $\ell$-th  Veronese  immersion of $K$.  We begin by providing a more explicit description of  the  Veronese map
 \[
 \eV_\ell: \underbrace{\bsU\otimes \cdots\otimes\bsU}_\ell\ra \eQ_\ell(\bsU)
 \]
 For any  
 $\ell$-linear form
 \[
 \Psi: \bsU\times\cdots\times \bsU\ra \bR 
 \]
  we define its symmetrization $\boldsymbol{\lan}\,\Psi\,\boldsymbol{\ran}\in\eQ_\ell(\bsU)$ by
  \[
  \boldsymbol{\lan}\,\Psi\,\boldsymbol{\ran}(\bu_1,\dotsc,\bu_\ell):=\frac{1}{\ell!}\sum_{\si\in\mathfrak{S}_\ell} \Psi(\bu_{\si(1)},\dotsc,\bu_{\si(\ell)}),
  \]
  where $\mathfrak{S}_\ell$ denotes the group of permutations of $\{1,\dotsc,\ell\}$. If $\bu_1,\dotsc, \bu_\ell\in \bsU$, we denote by $\bu_i^\dag$ the linear functionals 
\[
\bu_i^\dag:\bsU\ra \bR,\;\;\bu_i^\dag(\bu)=(\bu_i,\bu)_\bsU.
\]
Then
\[
\eV_\ell(\bu_1,\dotsc,\bu_\ell)= \boldsymbol{\bigl\lan}\,\bu_1^\dag\otimes\cdots\otimes\bu_\ell^\dag\,\boldsymbol{\bigr\ran}.
\]
Denote by $(e_1,e_2)$ the canonical orthonormal  basis of $\bR^2$  and by $(e^1,e^2)$ its dual basis.    An \emph{orthonormal basis} of $\eQ_\ell$ is given by the monomials
\[
\Phi_j= \binom{\ell}{j}^{1/2}\boldsymbol{\Bigl\lan}\, (e^1)^{\otimes j}\otimes (e^2)^{\otimes (\ell-j)}\,\boldsymbol{\Bigr\ran},\;\;0\leq j\leq \ell.
\]
For any $\bx\in\bsU$ we have
\[
\eV_\ell(\bx)= \sum_{j=0}^\ell V_j(\bx)\Phi_j,
\]
where
\[
V_j(x)= \lan \eV_\ell(\bx),\Phi_j\ran_\ell =\Phi_j(\bx,\dotsc,\bx)= \binom{\ell}{j}^{1/2}  x_1^jx_2^{\ell-j}.
\]
The map $\eV_\ell$ has an important invariance property.  Observe that  we have a linear    right action of  $SO(\bsU)$ on $\eQ_\ell$ 
\[
\eQ_\ell(\bsU)\times SO(\bsU)\ni (\Phi,R)\mapsto \eT_R\Phi\in \eQ_\ell(\bsU),
\]
where $\eT_R:\eQ_\ell(\bsU)\ra \eQ_\ell(\bsU)$ is the \emph{isometry} defined  by
\[
(\eT_R\Phi) (\bu_1,\dotsc, \bu_\ell)=\Phi(R\bu_1,\dotsc, R\bu_\ell),
\]
for all $\Phi\in\eQ_\ell$ and  $\bu_1,\dotsc,\bu_\ell\in \bsU$.   The   equality (\ref{eq: ver0}) implies that  for any $\bu_1,\dotsc,\bu_\ell\in\bsU$ we have  
\begin{equation}
\eV_\ell(R\bu_1,\dotsc,R\bu_\ell)=\eT_R^\dag \bigl(\eV_\ell(\bu_1,\dotsc,\bu_\ell)\,\bigr),
\label{eq:  action}
\end{equation}
where $\eT_R^\dag:\eQ_\ell(\bsU)\ra \eQ_\ell(\bsU)$ denotes the  adjoint of $\eT_R$. 

Consider the   standard parametrization $\theta\mapsto\bx(\theta)=(\cos\theta,\sin\theta)$ of the unit circle in $\bsU$.  The $\ell$-th  Veronese immersion of the unit circle is given by the map
\[
V: [0,2\pi]\ni\theta \mapsto \bigl(\, V_0(\theta),\dotsc, V_\ell(\theta)\,\bigr)\in \bR^{\ell+1}
\]
\[
V_j(\theta)=\binom{\ell}{j}^{1/2}(\cos\theta)^j(\sin\theta)^{(\ell-j)},\;\;0\leq j\leq \ell.
\]

We have
\[
V'(\theta)=\ell \eV_\ell\bigl(\,\bx'(\theta),\bx(\theta), \dotsc, \bx(\theta)\,\bigr),,\;\;\bx''(\theta)=-\bx(\theta).
\]
and
\[
V''(\theta)=\ell \eV_\ell\bigl(\,\bx''(\theta),\bx(\theta), \dotsc, \bx(\theta)\,\bigr)+ \ell(\ell-1) \eV_\ell\bigl(\,\bx'(\theta),\bx'(\theta), \bx(\theta),\dotsc, \bx(\theta)\,\bigr).
\]
If we denote by $R_\theta$ the counterclockwise rotation of $\bR^2$ of angle $\theta$, then 
\[
\bx(\theta)= R_\theta\bx(0),\;\;\bx'(\theta)= R_\theta \bx'(0). 
\]
Using (\ref{eq: action}) we deduce that
\[
V'(\theta)=\eT_{R_\theta}^\dag V'(0),\;\; V''(\theta)= \eT_{R_\theta}^\dag  V''(\theta).
\]
In particular, since the map  $\eT_R^\dag:\eQ_\ell(\bsU)\ra \eQ_\ell(\bsU)$ is an isometry we deduce
\[
|V'(\theta)|=|V'(0)|,\;\;|V''(\theta)|= |V''(0)|,\;\;\forall \theta.
\]
Observe that for $0<j <\ell$ we have
\[
V'_j(\theta)=\binom{\ell}{j}^{1/2}\left( -j(\cos\theta)^{(j-1)}(\sin\theta)^{(\ell-j+1)}+(\ell-j)(\cos\theta)^{(j+1)}(\sin\theta)^{(\ell-j-1)}\,\right)
\]
\[
=\binom{\ell}{j}^{1/2}(\cos\theta)^{(j-1)}(\sin\theta)^{(\ell-j-1)}\bigl(\,-j\sin^2\theta+(\ell-j)\cos^2\theta\,\bigr)
\]
\[
=\binom{\ell}{j}^{1/2}(\cos\theta)^{(j-1)}(\sin\theta)^{(\ell-j-1)}\bigl(-j +\ell\cos^2\theta).
\]
Next we observe that
\[
V_\ell'(\theta)= -\ell (\cos\theta)^{\ell-1}\sin\theta,\;\; V_0'(\theta)=\ell(\sin\theta)^{\ell-1}\cos\theta.
\]
This shows that for $\theta=0$ only the component $V'_{\ell-1}(0)$ is non trivial and it is equal to $\ell^{1/2}$. This proves that 
\[
|V'(\theta)|=\ell^{1/2}.
\]
Thus the arclength parameter along the  curve $\theta\mapsto V(\theta)$ is $s=\ell^{1/2}\theta$, the length of $K_\ell$ is $2\pi\ell^{1/2}$  and
\[
\frac{d}{ds} = \ell^{-1/2}\frac{d}{d\theta}.
\]
Next we compute $|V''(0)|$.   If $\ell-j-1\geq 2$, i.e. $j\leq \ell-3$ we have $V_j''(0)=0$.
\[
V_\ell''(0)=-\ell,  \;\; V_{\ell-1}''(0)=0,\;\;V''_{\ell-2}=2\binom{\ell}{2}^{1/2}.
\]
Hence
\[
\left|\frac{d^2}{d\theta^2}V(\theta)\right|^2= \ell^2 +4\binom{\ell}{2} = 3\ell^2-2\ell \Rightarrow \left|\frac{d^2}{ds^2}V(s)\right|^2=\frac{3\ell-2}{\ell}.
\]
Thus
\[
\tau_\ell(K)= \frac{1}{\pi} \int_0^{2\pi\ell^{1/2} } \Biggl|\frac{d^2}{ds^2} V(s)\,\Biggr| ds= 2\sqrt{3\ell-2}.
\]
We have thus proved the following result.

\begin{theorem} If $K$ is the unit circle in  $\bR^2$ centered at the origin, then
\[
\mu_K^\ell=2\sqrt{3\ell-2},\;\;\forall \ell\geq 1.\proofend
\]
\label{th:  typical}
\end{theorem}

\begin{remark} (a) The case $\ell=1$ of the above theorem is   obvious.  For the case $\ell=2$, i.e., the equality $\mu_K^2=4$, we can give two alternate proofs.     For the first proof we  observe that
\[
\eV_2(\cos\theta,\sin\theta) =\bigl(\, x\theta),y(\theta),z(\theta)\,\bigr) (\sin^2\theta, \sqrt{2}\sin\theta \cos\theta,\cos^2\theta).
\]
Observe that  the image of the   $2$nd Veronese immersion is the intersection of the sphere $x^2+y^2+z^2=1$ with the plane $x+z=1$ which is a circle of radius $\frac{1}{\sqrt{2}}$.  Moreover, the Veronese  immersion double covers this circle. These facts imply easily that $\mu_K^2=4$.

For the second proof  we observe that the typical pencil of plane conics
\[
C_t=\{ ax^2+2bxy+c^2y^2=t\}
\]
 has  four points of tangency with the  unit circle.

 \begin{figure}[ht]
\centering{\includegraphics[height=2.5in,width=2.5in]{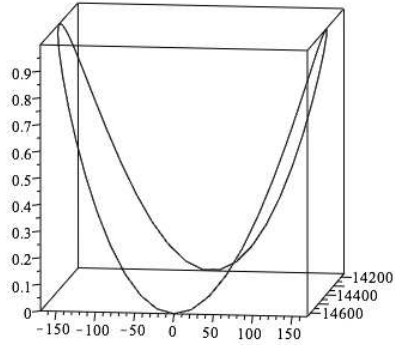}}
\caption{\sl The  second Veronese embedding of the circle of radius $1$ with center $(120,0)$.}
\label{fig: 1}
\end{figure}
 
 (b) The invariant $\mu_K^\ell$ is invariant under rotations, but not under translations. For example if $K$ is the    circle of radius $1$ in the plane centered at the point $(3,0)$. Then the second Veronese immersion of $K$ is given by
 \[
 [0,2\pi]\ni t\mapsto \bigl(\, (3+\cos t)^2, \sqrt{2}(3+\cos t)\sin t, \sin^2 t\,\bigr)\in\bR^3.
 \]
 \emph{MAPLE} aided computation shows that its pointwise curvature is
 \[
\kappa(t)= \sqrt {-{\frac {108\, \left( \cos \left( t \right)  \right) ^{3}-54\,
 \left( \cos \left( t \right)  \right) ^{2}-360\,\cos \left( t
 \right) -238}{ -2\left( 9\, \left( \cos \left( t \right)  \right) ^{2}-
19-6\,\cos \left( t \right)  \right) ^{3}}}},
\]  
while its total curvature is
\[
\mu_K^2\approx 2.065 < 2\sqrt{3\cdot 2-2}=4.
\]
Numerical experimentation suggest  that if $K(d)$ is the circle of radius $1$ centered at the point $(d,0)$ then
\[
\lim_{d\ra \infty} \mu^2_{K(d)} =2.
\]
In view of Fenchel theorem, this would indicate that  for $d$ large  the second  Veronese  embedding of $K(d)$ is very close to a planar convex curve.   The \emph{MAPLE} generated   Figure \ref{fig: 1}   seems to confirm this, since the second Veronese immersion   is contained in a box  of dimensions $1\times a\times b$, where $a>300$ and $b>400$.\qed

 \end{remark}

\section{Critical points of trigonometric polynomials}
\setcounter{equation}{0}

As another application of Theorem \ref{th: main}, we want to compute the expected number of critical points on the unit circle of a trigonometric polynomial of degree $\leq n$.

Let $K$ denote the     unit circle in $\bR^2$ centered at the origin and denote by $\bsV_n$ the vector space of trigonometric polynomials of degree $\leq n$ on $K$, i.e.,     functions $f:K\ra \bR$ of the form
\[
f(\theta)= a_0+\sum_{k=1}^n  \bigl(\,a_k\cos k\theta+b_k\sin k\theta\,\bigr),
\]
equipped with the inner product
\[
(f,g)_\bsV= \frac{1}{2\pi}\int_0^{2\pi}  f(\theta) g(\theta) d\theta.
\]
Thus
\[
\left| a_0+\sum_{k=1}^n  \bigl(\,a_k\cos k\theta+b_k\sin k\theta\,\bigr)\,\right|^2_\bsV= a_0^2 + \frac{1}{2}\sum_{k=1}^n (a_k^2+b_k^2).
\]
We  obtain  an orthonormal basis of $\bsV_n$ 
\[
\be_0=1,\;\;\be_{k}=\sqrt{2} \cos k\theta,\;\;\bsf_k= \sqrt{2} \sin k\theta ,\;\;k=1,\dotsc, n.
\]
The function
\[
\xi: K\ra \bsV_n,\;\;[0,2\pi] \ni s \mapsto \xi_s^\dag
\]
defined by  (\ref{eq: xi}) +(\ref{eq: xi-d}) is given in this case by  the map
\[
s\mapsto \xi_s^\dag= \sqrt{2}\sum_{k=1}^n k\bigl(-(\sin ks)\be_k +(\cos ks)\bsf_k\,\bigr).
\]
Observe that
\[
|\xi_s^\dag|^2= 2\sum_{k=1}^n k^2.
\]
We set
\[
\bnu(s):=\frac{1}{|\xi_s^\dag|}\xi_s^\dag=\left(\sum_{k=1}^n k^2\right)^{-1/2} \sum_{k=1}^n k\bigl(-(\sin ks)\be_k +(\cos ks)\bsf_k\,\bigr).
\]
Then
\[
\bnu'(s)= \left(\sum_{k=1}^n k^2\right)^{-1/2} \sum_{k=1}^n k^2\bigl(\, (\cos ks)\be_k -(\sin ks)\bsf_k\,\bigr),
\]
so that
\[
|\bnu'(s)|= \left(\sum_{k=1}^n k^2\right)^{-1/2}\cdot \left(\sum_{k=1}^n k^4\right)^{1/2} 
\]
Using the classical identity
\[
\sum_{k=1}^n k^j =\frac{1}{j+1}\beta_{j+1}(n+1),
\]
where $\beta_\ell$ is the $\ell$-th Bernoulli polynomial   we deduce
\begin{equation}
|\bnu'(s)|=\sqrt{\frac{3\beta_5(n+1)}{5\beta_3(n+1)}},
\label{eq: trig}
\end{equation}
where
\[
\beta_3(x)= x^3-\frac{3}{2}x^2+\frac{1}{2}x,\;\;\beta_5(x)=x^5-\frac{5}{2}x^4 +\frac{5}{3}x^3-\frac{1}{6}x.
\]
The equality (\ref{eq: trig}) and Theorem \ref{th: main} imply the following result. 

\begin{theorem} The expected number of critical points on the unit circle of a  trigonometric polynomial of degree $\leq n$ is
\[
T_n=2 \sqrt{\frac{3\beta_5(n+1)}{5\beta_3(n+1)}}=2\sqrt{\frac{\sum_{k=1}^nk^4}{\sum_{k=1}^nk^2}}.
\]
In particular  
\[
T_n \sim 2n\sqrt{\frac{3}{5}} \;\;\mbox{as $n\ra \infty$}.\proofend
\]
\label{th: trig}
\end{theorem}

\begin{remark}     The above theorem  indicates  that for large $k$, a trigonometric  polynomial  of degree $k$  is expected to have approximately  $ 2\sqrt{\frac{3}{5}} k\approx 1.549 k$ critical points.  In particular, it will likely have less that $2k$ critical points on the unit circle.

\begin{figure}[ht]
\centering{\includegraphics[height=2.5in,width=2.5in]{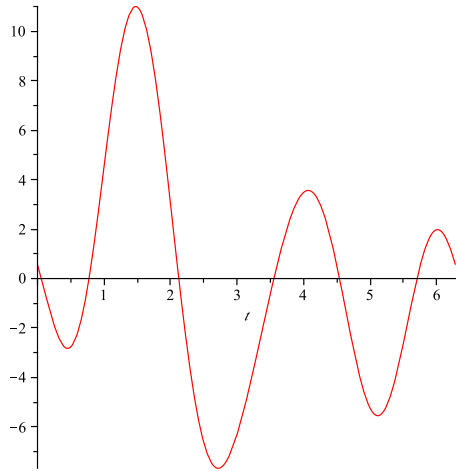}}
\caption{\sl A ``random'' trigonometric polynomial of degree $4$.}
\label{fig: trig}
\end{figure}

In the \emph{MAPLE} generated Figure  \ref{fig: trig} we depicted the graph of  a ``random''  trigonometric polynomial of degree $4$
\[
1.2\cos t  + 2.35\sin t  - 3.17\cos 2t + 2.71\sin 2t + 1.53\cos 3t  - 4.17\sin 3t 
 \]
 \[
 +\cos 4t  - 1.15\sin 4t .
\]
 We notice  that it has  $6$ critical points, which is very close to the  expected value $T_4\approx 6.87$. \qed \end{remark}

\end{document}